\documentclass[10pt]{amsart}
\usepackage{amsmath,amscd}
\usepackage{amsbsy}
\usepackage{amssymb}
\usepackage{amscd,amsthm}

\usepackage[all,cmtip]{xy}

\newtheorem{thm}{Theorem}\numberwithin{thm}{section}
\newtheorem{lem}[thm]{Lemma}
\newtheorem{prop}[thm]{Proposition}

\newtheorem{exam}[thm]{Example}
\newtheorem{rema}[thm]{Remark}

\newtheorem{defi}[thm]{Definition}

\newtheorem*{thm2}{Theorem}

\begin{document}
\begin{center}
\huge{Equivariant derived category of flat families}\\[1cm]
\end{center}
\begin{center}

\large{Sa$\mathrm{\check{s}}$a Novakovi$\mathrm{\acute{c}}$}\\[0,5cm]
\end{center}
{\small \textbf{Abstract}. 
We prove the existence of tilting bundles on global quotient stacks that are produced by compatible finite group actions on flat families. 

\begin{center}
\tableofcontents
\end{center}

\section{Introduction}
Geometric tilting theory started with the construction of tilting bundles on the projective space by Beilinson \cite{BE}. Later Kapranov \cite{KA}, \cite{KA1}, \cite{KA2} constructed tilting bundles for certain homogeneous spaces. Further examples can be obtained from certain blow ups and taking projective bundles \cite{CM}, \cite{CDR}, \cite{O}. A smooth projective $k$-scheme admitting a tilting object satisfies very strict conditions, namely its Grothendieck group is a free abelian group of finite rank and the Hodge diamond is concentrated on the diagonal, at least in characteristic zero \cite{BH}.

However, it is still an open problem to give a complete classification of smooth projective $k$-schemes admitting a tilting object. In the case of curves one can prove that a smooth projective algebraic curve has a tilting object if and only if the curve is a one-dimensional Brauer--Severi variety. 
But already for smooth projective algebraic surfaces there is currently no classification of surfaces admitting such a tilting object. It is conjectured that a smooth projective algebraic surface has a tilting bundle if and only if it is rational (see \cite{BSH}, \cite{H}, \cite{HP}, \cite{HP1}, \cite{HP2}, \cite{KI} and \cite{P} for results in this direction).

In the present work, we will focus on a certain type of a quotient stack and prove the existence of tilting bundles for their derived category. Several examples of stacks admitting a tilting object are known (see \cite{IU}, \cite{IU1}, \cite{KAW}, \cite{ME}, \cite{NO}, \cite{OU} and \cite{OUE}). But as in the case of schemes, one has to settle for existence criteria for stacks admitting a tilting object.

Assume $k$ is an algebraically closed field of characteristic zero. In \cite{NO1} it is proved a generalization of the main result of \cite{CDR}. It is the following theorem: 
\begin{thm2}(\cite{NO1}, Theorem 4.15)
Let $\pi\colon X\rightarrow Z$ be a flat proper morphism between smooth projective $k$-schemes and $\mathcal{E}_1,...,\mathcal{E}_n$ a set of locally free sheaves in $D^b(X)$ such that for any point $z\in Z$ the collection $\mathcal{E}^z_1=\mathcal{E}_1\otimes \mathcal{O}_{X_z},...,\mathcal{E}^z_n=\mathcal{E}_n\otimes \mathcal{O}_{X_z}$ of the restrictions to the fiber $\pi^{-1}(z)=X_z$ is a full strongly exceptional collection for $D^b(X_z)$. Suppose $\mathcal{T}$ is a tilting bundle on $Z$. Then there exists an ample sheaf $\mathcal{M}$ on $Z$ such that $\bigoplus^n_{i=1}\pi^*(\mathcal{T}\otimes \mathcal{M}^{\otimes i})\otimes \mathcal{E}_i$ is a tilting bundle on $X$.
\end{thm2}
Our goal is to give an equivariant version of this theorem and in this way to produce a tilting bundle on the quotient stack $[X/G]$. So assume $k$ is algebraically closed and of characteristic zero. We consider smooth projective $k$-schemes $X$ and $Z$ with a compatible action of a finite group $G$ and investigate the case of $G$-morphisms $\pi\colon X\rightarrow Z$ where the underlying morphism of schemes $X\rightarrow Z$ is flat and proper. We call such morphisms \emph{flat $G$-maps} for simplicity. For a smooth projective $k$-scheme $X$, denote by $\mathrm{Coh}_G(X)$ the abelian category of equivariant coherent sheaves and by $D^b_G(X)$ its bounded derived category. Furthermore, $k[G]$ denotes the regular representation of $G$. We then prove the following theorem. 

\begin{thm2}(Theorem 4.5)
Let $\pi\colon X\rightarrow Z$ be a flat $G$-map and $\mathcal{E}_1,...,\mathcal{E}_n$ a set of locally free sheaves in $D^b_G(X)$ such that, considered as a set of objects in $D^b(X)$, for any point $z\in Z$ the collection $\mathcal{E}^z_1=\mathcal{E}_1\otimes \mathcal{O}_{X_z},...,\mathcal{E}^z_n=\mathcal{E}_n\otimes \mathcal{O}_{X_z}$ of the restrictions to the fiber $\pi^{-1}(z)=X_z$ is a full strongly exceptional collection for $D^b(X_z)$. Suppose $\mathcal{T}\in \mathrm{Coh}_G(Z)$ is a tilting bundle on $Z$. There exists an equivariant ample sheaf $\mathcal{M}$ on $Z$ such that $(\bigoplus^n_{i=1}\pi^*(\mathcal{T}\otimes \mathcal{M}^{\otimes i})\otimes \mathcal{E}_i)\otimes k[G]$ is a tilting bundle on $[X/G]$.
\end{thm2}
If all indecomposable pairwise non-isomorphic direct summands of $\mathcal{T}$ and all $\mathcal{E}_i$ are invertible sheaves, we obtain a full strongly exceptional collection. 
\begin{thm2}(Theorem 4.9)
Let $\pi\colon X\rightarrow Z$ be a flat $G$-map and $\mathcal{E}_1,...,\mathcal{E}_n$ a set of invertible sheaves in $D^b_G(X)$ such that, considered as a set of objects in $D^b(X)$, for any point $z\in Z$ the collection $\mathcal{E}^z_1=\mathcal{E}_1\otimes \mathcal{O}_{X_z},...,\mathcal{E}^z_n=\mathcal{E}_n\otimes \mathcal{O}_{X_z}$ of the restrictions to the fiber $\pi^{-1}(z)=X_z$ is a full strongly exceptional collection for $D^b(X_z)$. Suppose $\mathcal{T}\in \mathrm{Coh}_G(Z)$ is a tilting bundle on $Z$ whose indecomposable pairwise non-isomorphic direct summands are invertible sheaves. Then there is a full strongly exceptional collection for $D^b_G(X)$.
\end{thm2}

{\small \textbf{Conventions}. Throughout this work $k$ is an algebraically closed field of characteristic zero and all locally free sheaves are assumed to be of finite rank.

\section{Generalities on equivariant derived categories}
Let $X$ be a quasiprojective $k$-scheme and $G$ a finite group acting on $X$. 
A \emph{G-linearization}, also called an \emph{equivariant structure}, on $\mathcal{F}$ is given by isomorphisms $\lambda_g\colon \mathcal{F}\stackrel{\sim}\rightarrow g^*\mathcal{F}$ for all $g\in G$ subject to $\lambda_1=\mathrm{id}_{\mathcal{F}}$ and $\lambda_{gh}=h^*\lambda_g\circ \lambda_h$. In the present work we also call such sheaves \emph{equivariant sheaves}. Equivariant sheaves are therefore pairs $(\mathcal{F}, \lambda)$, consisting of a sheaf $\mathcal{F}$ on $X$ and a choice of an equivariant structure $\lambda$. 
\begin{rema}
\textnormal{
For a definition of linearization in the case where an arbitrary algebraic group acts on an arbitrary scheme we refer to \cite{BL}, \cite{EL1} or \cite{EL2}.}
\end{rema}  
If $(\mathcal{F},\lambda)$ and $(\mathcal{G},\mu)$ are two equivariant sheaves on $X$, the vector space $\mathrm{Hom}(\mathcal{F},\mathcal{G})$ becomes a $G$-representation via $g\cdot f:=(\mu_g)^{-1}\circ g^*f\circ \lambda_g$ for $f\colon \mathcal{F}\rightarrow \mathcal{G}$. The equivariant quasi-coherent respectively coherent sheaves together with $G$-invariant morphisms $\mathrm{Hom}_G(\mathcal{F},\mathcal{G}):=\mathrm{Hom}(\mathcal{F},\mathcal{G})^G$ form abelian categories with enough injectives (see \cite{BR}, \cite{PL}) which we denote by $\mathrm{Qcoh}_G(X)$ respectively $\mathrm{Coh}_G(X)$. We put $D_G(\mathrm{Qcoh}(X)):=D(\mathrm{Qcoh}_G(X))$ and $D^b_G(X):=D^b(\mathrm{Coh}_G(X))$. 

Let $X$ and $Y$ be quasiprojective $k$-schemes on which the finite group $G$ acts. A \emph{$G$-morphisms} between $X$ and $Y$ is given by a morphism $\phi\colon X\rightarrow Y$ such that $\phi\circ g=g\circ \phi$ for all $g\in G$. Then we have the pullback $\phi^*\colon \mathrm{Coh}_G(Y)\rightarrow \mathrm{Coh}_G(X)$ and the pushforward $\phi_*\colon \mathrm{Coh}_G(X)\rightarrow \mathrm{Coh}_G(Y)$. The functors $\phi^*$ and $\phi_*$ are adjoint; analogously for $\mathbb{L}\phi^*$ and $\mathbb{R}\phi_*$. For $(\mathcal{F},\lambda), (\mathcal{G},\mu)\in \mathrm{Coh}_G(X)$ there is a canonical equivariant structure on $\mathcal{F}\otimes\mathcal{G}$ coming from the maps $\lambda_g\otimes\mu_g$ (see \cite{BL}, Proposition 3.46).

By definition, objects of $D^b_G(X)$ are bounded complexes of equivariant coherent sheaves. It is clear that each such complex defines an equivariant structure on the corresponding object of $D^b(X)$. Now let $\mathcal{C}$ be the category of equivariant objects of $D^b(X)$, i.e. complexes $\mathcal{F}^{\bullet}$ with isomorphisms $\lambda_g\colon \mathcal{F}^{\bullet}\stackrel{\sim}\rightarrow g^*\mathcal{F}^{\bullet}$ satisfying $\lambda_{gh}=h^*\lambda_g\circ \lambda_h$. This category is in fact triangulated and it is a natural fact that $D^b_G(X)$ and $\mathcal{C}$ are equivalent (see \cite{C}, Proposition 4.5 or \cite{EL2}). 

There is also another description of the derived categories needed in the present work. Consider the global quotient stack $[X/G]$, produced by an action of a finite group $G$ on $X$ (see \cite{V}, Example 7.17). The quasi-coherent sheaves on $[X/G]$ are equivalent to equivariant quasi-coherent sheaves on $X$ (see \cite{V}, Example 7.21). Henceforth, the abelian categories $\mathrm{Qcoh}([X/G])$ and $\mathrm{Qcoh}_G(X)$ are equivalent and therefore give rise to equivalent derived categories $D_G(\mathrm{Qcoh}(X))\simeq D(\mathrm{Qcoh}([X/G]))$. For any two objects $\mathcal{F}^{\bullet}, \mathcal{G}^{\bullet}\in D_G(\mathrm{Qcoh}(X))$ we write $\mathrm{Hom}_G(\mathcal{F}^{\bullet},\mathcal{G}^{\bullet}):=\mathrm{Hom}_{D_G(\mathrm{Qcoh}(X))}(\mathcal{F}^{\bullet},\mathcal{G}^{\bullet})$.

Analogously, we get $D^b_G(X)\simeq D^b(\mathrm{Coh}([X/G]))$. 
Note that for $X=pt$, $\mathrm{Coh}([pt/G])\simeq \mathrm{Coh}_G(pt)\simeq \mathrm{Rep}_k(G)$ is the category of finite-dimensional representations. Moreover, for a finite group $G$, the functor $(-)^G\colon \mathrm{Coh}([pt/G])\rightarrow \mathrm{Coh}(pt), V\mapsto V^G$, is exact (see \cite{AOV}, Proposition 2.5). For arbitrary $\mathcal{F}^{\bullet}, \mathcal{G}^{\bullet}\in D^b_G(X)$, the finite group $G$ also acts on the vector space $\mathrm{Hom}(\mathcal{F}^{\bullet},\mathcal{G}^{\bullet}):=\mathrm{Hom}_{D^b(X)}(\mathcal{F}^{\bullet},\mathcal{G}^{\bullet})$. The exactness of $(-)^G$ yields
\begin{center}
$\mathrm{Hom}_G(\mathcal{F}^{\bullet},\mathcal{G}^{\bullet})\simeq \mathrm{Hom}(\mathcal{F}^{\bullet},\mathcal{G}^{\bullet})^G$. 
\end{center}
The exactness of $(-)^G$ also implies the following fact (see \cite{BFK}, Lemma 2.2.8):
\begin{lem}
Let $X$ be smooth quasiprojective $k$-scheme and $G$ a finite group acting on $X$. For arbitrary $\mathcal{F}^{\bullet},\mathcal{G}^{\bullet}\in D^b_G(X)$ the following holds for all $i\in \mathbb{Z}$:
\begin{center}
$\mathrm{Hom}_{G}(\mathcal{F}^{\bullet},\mathcal{G}^{\bullet}[i])\simeq \mathrm{Hom}(\mathcal{F}^{\bullet},\mathcal{G}^{\bullet}[i])^G$.
\end{center}
\end{lem} 
For $\mathcal{F}\in \mathrm{Coh}_G(X)$ we therefore have $H^i_G(X,\mathcal{F})\simeq H^i(X,\mathcal{F})^G$. In Section 4 we also need the Leray spectral sequence: For a $G$-morphism $f\colon X\rightarrow Y$, the spectral sequence is
\begin{eqnarray}
E^{p,q}_2=H^p_G(Y,\mathbb{R}^q f_*(\mathcal{F}^{\bullet}))\Longrightarrow H^{p+q}_G(X,\mathcal{F}^{\bullet}).
\end{eqnarray}
\section{Geometric tilting theory}
In this section we recall some facts of geometric tilting theory. We first recall the notions of generating and thick subcategories (see \cite{BV}, \cite{RO}).\\

Let $\mathcal{D}$ be a triangulated category and $\mathcal{C}$ a triangulated subcategory. The subcategory $\mathcal{C}$ is called \emph{thick} if it is closed under isomorphisms and direct summands. For a subset $A$ of objects of $\mathcal{D}$ we denote by $\langle A\rangle$ the smallest full thick subcategory of $\mathcal{D}$ containing the elements of $A$. 
Furthermore, we define $A^{\perp}$ to be the subcategory of $\mathcal{D}$ consisting of all objects $M$ such that $\mathrm{Hom}_{\mathcal{D}}(E[i],M)=0$ for all $i\in \mathbb{Z}$ and all elements $E$ of $A$. We say that $A$ \emph{generates} $\mathcal{D}$ if $A^{\perp}=0$. Now assume $D$ admits arbitrary direct sums. An object $B$ is called \emph{compact} if $\mathrm{Hom}_{\mathcal{D}}(B,-)$ commutes with direct sums. Denoting by $\mathcal{D}^c$ the subcategory of compact objects we say that $\mathcal{D}$ is \emph{compactly generated} if the objects of $\mathcal{D}^c$ generate $\mathcal{D}$. One has the following important theorem (see \cite{BV}, Theorem 2.1.2).
\begin{thm}
Let $\mathcal{D}$ be a compactly generated triangulated category. Then a set of objects $A\subset \mathcal{D}^c$ generates $\mathcal{D}$ if and only if $\langle A\rangle=\mathcal{D}^c$.  
\end{thm}
We now give the definition of tilting objects (see \cite{BH} for a definition of tilting objects in arbitrary triangulated categories).
\begin{defi}
\textnormal{Let $k$ be a field, $X$ a quasiprojective $k$-scheme and $G$ a finite group acting on $X$. An object $\mathcal{T}^{\bullet}\in D_G(\mathrm{Qcoh}(X))$ is called \emph{tilting object} on $[X/G]$ 
if the following hold:
\begin{itemize}
      \item[\bf (i)] Ext vanishing: $\mathrm{Hom}_G(\mathcal{T}^{\bullet},\mathcal{T}^{\bullet}[i])=0$ for $i\neq0$.
     \item[\bf (ii)] Generation: If $\mathcal{N}^{\bullet}\in D_G(\mathrm{Qcoh}(X))$ satisfies $\mathbb{R}\mathrm{Hom}_G(\mathcal{T}^{\bullet},\mathcal{N}^{\bullet})=0$, then $\mathcal{N}^{\bullet}=0$.
		\item[\bf (iii)] Compactness: $\mathrm{Hom}_G(\mathcal{T}^{\bullet},-)$ commutes with direct sums.
		\end{itemize}}
\end{defi}
Below we state the well-known equivariant tilting correspondence (see \cite{BR}, Theorem 3.1.1). It is a direct application of a more general result on triangulated categories (see \cite{KE}, Theorem 8.5). We denote by $\mathrm{Mod}(A)$ the category of right $A$-modules and by $D^b(A)$ the bounded derived category of finitely generated right $A$-modules. Furthermore, $\mathrm{perf}(A)\subset D(\mathrm{Mod}(A)) $ denotes the full triangulated subcategory of perfect complexes, those quasi-isomorphic to a bounded complexes of finitely generated projective right $A$-modules.
\begin{thm}
Let $X$ be a quasiprojective $k$-scheme and $G$ a finite group acting on $X$. Suppose we are given a tilting object $\mathcal{T}^{\bullet}$ on $[X/G]$ 
and let $A=\mathrm{End}_G(\mathcal{T}^{\bullet})$. Then the following hold:
\begin{itemize}
      \item[\bf (i)] The functor $\mathbb{R}\mathrm{Hom}_G(\mathcal{T}^{\bullet},-)\colon D_G(\mathrm{Qcoh}(X))\rightarrow D(\mathrm{Mod}(A))$ is an equivalence. 
      \item[\bf (ii)] If $X$ is smooth and $\mathcal{T}\in D^b_G(X)$, this equivalence restricts to an equivalence $D^b_G(X)\stackrel{\sim}\rightarrow \mathrm{perf}(A)$.
			
			\item[\bf (iii)] If the global dimension of $A$ is finite, then $\mathrm{perf}(A)\simeq D^b(A)$. 
\end{itemize} 
\end{thm}
 \begin{rema}
\textnormal{If $X$ is a smooth projective $k$-scheme and $G=1$, the derived category $D(\mathrm{Qcoh}(X))$ is compactly generated and the compact objects are exactly $D^b(X)$ (see \cite{BV}). In this case, a compact object $\mathcal{T}^{\bullet}$ generates $D(\mathrm{Qcoh}(X))$ if and only if $\langle\mathcal{T}^{\bullet}\rangle=D^b(X)$. Since the natural functor $D^b(X)\rightarrow D(\mathrm{Qcoh}(X))$ is fully faithful (see \cite{HUY}), a compact object $\mathcal{T}^{\bullet}\in D(\mathrm{Qcoh}(X))$ is a tilting object if and only if $\langle\mathcal{T}^{\bullet}\rangle=D^b(X)$ and $\mathrm{Hom}_{D^b(X)}(\mathcal{T}^{\bullet},\mathcal{T}^{\bullet}[i])=0$ for $i\neq 0$. If the tilting object $\mathcal{T}^{\bullet}$ is a coherent sheaf and $\mathrm{gldim}(\mathrm{End}(\mathcal{T}^{\bullet}))<\infty$, the above definition coincides with the definition of a tilting sheaf given in \cite{B}. In this case the tilting object is called \emph{tilting sheaf} on $X$. If it is a locally free sheaf we simply say that $\mathcal{T}$ is a \emph{tilting bundle}. Theorem 3.3 then gives the classical tilting correspondence as first proved by Bondal \cite{BO} and later extended by Baer \cite{B}}.
\end{rema}
The next observation shows that in Theorem 3.3 the smoothness of $X$ already implies the finiteness of the global dimension of $A$.
\begin{prop}
Let $X$, $G$ and $\mathcal{T}^{\bullet}$ be as in Theorem 3.3. If $X$ is smooth and projective, then $A=\mathrm{End}_G(\mathcal{T}^{\bullet})$ has finite global dimension and therefore the equivalence (i) of Theorem 3.3 restricts to an equivalence $D^b_G(X)\stackrel{\sim}\rightarrow D^b(A)$. 
\end{prop}
\begin{proof}
Imitating the proof of Theorem 7.6 in \cite{HV}, we argue as follows: For two finitely generated right $A$-modules $M$ and $N$, the equivalence $\psi:=\mathbb{R}\mathrm{Hom}_G(\mathcal{T}^{\bullet},-)\colon D^b_G(X)\rightarrow \mathrm{perf}(A)$ (see Theorem 3.3 (ii)) yields
\begin{center}
$\mathrm{Ext}^i_A(M,N)\simeq \mathrm{Hom}_G(\psi^{-1}(M),\psi^{-1}(N)[i])\simeq \mathrm{Hom}(\psi^{-1}(M),\psi^{-1}(N)[i])^G=0$ 
\end{center}
for $i\gg 0$, since $X$ is smooth. Indeed, this follows from the local-to-global spectral sequence, Grothendieck vanishing Theorem and Lemma 2.2. As $X$ is projective, $A=\mathrm{End}_G(\mathcal{T}^{\bullet})$ is a finite-dimensional $k$-algebra and hence a noetherian ring. But for noetherian rings the vanishing of $\mathrm{Ext}^i_A(M,N)$ for $i\gg 0$ for any two finitely generated $A$-modules $M$ and $N$ suffices to conclude that the global dimension of $A$ has to be finite.   
\end{proof}
The following fact is folklore. It will be needed in Section 4.
\begin{prop}
Let $X$ be a smooth projective $k$ and $\mathcal{T}^{\bullet}$ a tilting object on $X$. Then for an invertible sheaf $\mathcal{L}$ the object $\mathcal{T}^{\bullet}\otimes \mathcal{L}$ is also a tilting object on $X$. 
\end{prop}
In the literature, instead of the tilting object $\mathcal{T}^{\bullet}$ one often studies the set $\mathcal{E}^{\bullet}_1,...,\mathcal{E}^{\bullet}_n$ of its indecomposable pairwise non-isomorphic direct summands. There is a special case where all the summands form a so-called full strongly exceptional collection. We recall the definition and follow here \cite{O2}.
\begin{defi}
\textnormal{Let $X$ and $G$ be as in Definition 3.2. An object $\mathcal{E}^{\bullet}\in D^b_G(X)$ is called \emph{exceptional} if $\mathrm{Hom}_G(\mathcal{E}^{\bullet},\mathcal{E}^{\bullet}[l])=0$ when $l\neq 0$, and $\mathrm{Hom}_G(\mathcal{E}^{\bullet},\mathcal{E}^{\bullet})=k$. An \emph{exceptional collection} in $D^b_G(X)$ is a sequence of exceptional objects $\mathcal{E}^{\bullet}_1,...,\mathcal{E}^{\bullet}_n$ satisfying $\mathrm{Hom}_G(\mathcal{E}^{\bullet}_i,\mathcal{E}^{\bullet}_j[l])=0$ for all $l\in \mathbb{Z}$ if $i>j$.}

\textnormal{The exceptional collection is called \emph{strongly exceptional} if in addition $\mathrm{Hom}_G(\mathcal{E}^{\bullet}_i,\mathcal{E}^{\bullet}_j[l])=0$ for all $i$ and $j$ when $l\neq 0$. Finally, we say the exceptional collection is \emph{full} if the smallest full thick subcategory containing all $\mathcal{E}^{\bullet}_i$ equals $D^b_G(X)$.}
\end{defi}
A generalization is the notion of a semiorthogonal decomposition of $D_G^b(X)$.
Recall, a full triangulated subcategory $\mathcal{D}$ of $D^b_G(X)$ is called \emph{admissible} if the inclusion $\mathcal{D}\hookrightarrow D^b_G(X)$ has a left and right adjoint functor.
\begin{defi}
\textnormal{Let $X$ and $G$ be as in Definition 3.2. A sequence $\mathcal{D}_1,...,\mathcal{D}_n$ of full triangulated subcategories of $D^b_G(X)$ is called \emph{semiorthogal} if all $\mathcal{D}_i\subset D^b_G(X)$ are admissible and $\mathcal{D}_j\subset \mathcal{D}_i^{\perp}=\{\mathcal{F}^{\bullet}\in D_G^b(X)\mid \mathrm{Hom}_G(\mathcal{G}^{\bullet},\mathcal{F}^{\bullet})=0$, $\forall$ $ \mathcal{G}^{\bullet}\in\mathcal{D}_i\}$ for $i>j$.}

\textnormal{Such a sequence defines a \emph{semiorthogonal decomposition} of $D^b_G(X)$ if the smallest full thick subcategory containing all $\mathcal{D}_i$ equals $D^b_G(X)$.}
\end{defi}

For a semiorthogonal decomposition of $D_G^b(X)$, we write $D_G^b(X)=\langle\mathcal{D}_1,...,\mathcal{D}_r\rangle$.
\begin{exam}
\textnormal{It is an easy exercise to show that a full exceptional collection $\mathcal{E}^{\bullet}_1,...,\mathcal{E}^{\bullet}_n$ in $D_G^b(X)$ gives rise to a semiorthogonal decomposition $D_G^b(X)=\langle\mathcal{D}_1,...,\mathcal{D}_n\rangle$ with $\mathcal{D}_i=\langle \mathcal{E}^{\bullet}_i\rangle$ (see \cite{HUY}, Example 1.60).} 
\end{exam}

Exceptional collections and semiorthogonal decompositions were intensively studied and we know quite a lot of examples of schemes admitting full exceptional collections or semiorthogonal decompositions. For an overview we refer to \cite{BOO} and \cite{KU}.

\section{Equivariant derived category of flat families}
Let $X$ and $Z$ be smooth projective $k$-schemes on which a finite group $G$ acts. In the sequel we consider $G$-morphisms $\pi\colon X\rightarrow Z$ where the underlying morphism of schemes $X\rightarrow Z$ is flat and proper. Such morphisms we simply call \emph{flat $G$-maps}.\\

We state some preliminary facts.
\begin{lem}
Let $\pi\colon X\rightarrow Z$ be a flat proper morphism and $\mathcal{E}_1,...,\mathcal{E}_n$ a set of locally free sheaves in $D^b(X)$ such that for any point $z\in Z$ the collection $\mathcal{E}^z_1=\mathcal{E}_1\otimes \mathcal{O}_{X_z},...,\mathcal{E}^z_n=\mathcal{E}_n\otimes \mathcal{O}_{X_z}$ of the restrictions to the fiber $\pi^{-1}(z)=X_z$ is a full strongly exceptional collection for $D^b(X_z)$. Then the following holds:
\[\mathbb{R}^s\pi_*(\mathcal{E}_q\otimes\mathcal{E}^{\vee}_p)=\begin{cases}
0& \text{for } s>0\\
0& \text{for s = 0 and } q<p\\
\pi_*(\mathcal{E}_q\otimes\mathcal{E}^{\vee}_p) &\text{for s = 0 and } q\geq p
\end{cases}\]
\end{lem}
\begin{proof}
If $\pi\colon X\rightarrow Z$ is a locally trivial fibration with typical fiber $F$ and $\mathcal{E}_1,...,\mathcal{E}_n$ are invertible sheaves this is exactly the claim of \cite{CDR}, p.430. As $\pi\colon X\rightarrow Z$ is flat and proper, flat base change holds (see \cite{HUY}, (3.18)). Carrying out the same arguments as in the proof of the claim of \cite{CDR}, p.430, we see that the above statement also holds for $\pi\colon X\rightarrow Z$ flat and proper and $\mathcal{E}_1,...,\mathcal{E}_n$ being arbitrary locally free sheaves.
\end{proof}
\begin{lem}
Let $X$ be a smooth projective $k$-scheme on which a finite group $G$ acts. Then there exists an equivariant ample sheaf $\mathcal{N}$.
\end{lem}
\begin{proof}
Let $\mathcal{L}$ be an ample invertible sheaf on $X$, then $g^*\mathcal{L}$ is ample for any $g\in G$. Now the tensor product $\bigotimes_{g\in G} g^*\mathcal{L}$ is ample and has a natural equivariant structure $\lambda$. Take $(\mathcal{N}, \lambda)=(\bigotimes_{g\in G} g^*\mathcal{L}, \lambda)$.
\end{proof}
\begin{lem}
Let $\pi\colon X\rightarrow Z$ and $\mathcal{E}_i$ be as in Lemma 4.1. Suppose $\mathcal{A}^{\bullet}$ is a compact object with $\langle \mathcal{A}^{\bullet}\rangle=D^b(Z)$, then $\langle\bigoplus^n_{i=1}\pi^*(\mathcal{A}^{\bullet})\otimes \mathcal{E}_i\rangle=D^b(X)$ and therefore $\bigoplus^n_{i=1}\pi^*(\mathcal{A}^{\bullet})\otimes \mathcal{E}_i$ generates $D(\mathrm{Qcoh}(X))$.
\end{lem}
\begin{proof}
In \cite{SAM}, Theorem 3.1 it is proved that the functor $\pi^*(-)\otimes \mathcal{E}_i\colon D^b(Z)\rightarrow D^b(X)$ is fully faithful and that $D^b(X)=\langle \pi^*D^b(Z)\otimes \mathcal{E}_1,...,\pi^*D^b(Z)\otimes \mathcal{E}_n\rangle$ is a semiorthogonal decomposition. Here the full subcategories $\pi^*D^b(Z)\otimes\mathcal{E}_i$ consist of objects of the form $\pi^*(\mathcal{F}^{\bullet})\otimes \mathcal{E}_i$, where $\mathcal{F}^{\bullet}\in D^b(Z)$. Therefore, the functor $\pi^*(-)\otimes \mathcal{E}_i$ from above induces an equivalence between $D^b(Z)$ and $\pi^*D^b(Z)\otimes\mathcal{E}_i$. Since $\langle\mathcal{A}^{\bullet}\rangle=D^b(Z)$, we immediately get $\langle\bigoplus^n_{i=1}\pi^*(\mathcal{A}^{\bullet})\otimes \mathcal{E}_i\rangle=D^b(X)$. Note that the compact objects of $D(\mathrm{Qcoh}(X))$ are all of $D^b(X)$ (see \cite{BV}). The rest follows from Theorem 3.1 as $\bigoplus^n_{i=1}\pi^*(\mathcal{A}^{\bullet})\otimes \mathcal{E}_i$ is a compact object of $D(\mathrm{Qcoh}(X))$. 
\end{proof}

To prove Theorem 4.5 below, we apply the following result that we only cite (see \cite{NO}, Theorem 4.1 and 4.2)
\begin{thm}
Let $X$ be a smooth projective $k$-scheme and $G$ a finite group acting on $X$. Suppose there is a $\mathcal{T}^{\bullet}\in D_G(\mathrm{Qcoh}(X))$ which, considered as an object in $D(\mathrm{Qcoh}(X))$, is a tilting object on $X$. Let $k[G]=\bigoplus_i W^{\oplus \mathrm{dim}(W_i)}_i$ be the regular representation of $G$, then $\mathcal{T}^{\bullet}\otimes k[G]$ and $\bigoplus_i\mathcal{T}^{\bullet}\otimes W_i$ are tilting objects on $[X/G]$.
\end{thm}

\begin{thm}
Let $\pi\colon X\rightarrow Z$ be a flat $G$-map and $\mathcal{E}_1,...,\mathcal{E}_n$ a set of locally free sheaves in $D^b_G(X)$ such that, considered as a set of objects in $D^b(X)$, for any point $z\in Z$ the collection $\mathcal{E}^z_1=\mathcal{E}_1\otimes \mathcal{O}_{X_z},...,\mathcal{E}^z_n=\mathcal{E}_n\otimes \mathcal{O}_{X_z}$ of the restrictions to the fiber $\pi^{-1}(z)=X_z$ is a full strongly exceptional collection for $D^b(X_z)$. Suppose $\mathcal{T}\in\mathrm{Coh}_G(Z)$ is a tilting bundle on $Z$. There is an equivariant ample sheaf $\mathcal{M}$ on $Z$ such that $(\bigoplus^n_{i=1}\pi^*(\mathcal{T}\otimes \mathcal{M}^{\otimes i})\otimes \mathcal{E}_i)\otimes k[G]$ is a tilting bundle on $[X/G]$.
\end{thm}
\begin{proof}
Below we show that there exists an equivariant ample sheaf $\mathcal{M}$ on $Z$ such that $(\bigoplus^n_{i=1}\pi^*(\mathcal{T}\otimes \mathcal{M}^{\otimes i})\otimes \mathcal{E}_i)\otimes k[G]$ is a tilting bundle on $[X/G]$. Note that by construction, $(\bigoplus^n_{i=1}\pi^*(\mathcal{T}\otimes \mathcal{M}^{\otimes i})\otimes \mathcal{E}_i)\otimes k[G]$ is a compact object of $D_G(\mathrm{Qcoh}(X))$. To apply Theorem 4.4, we only have to show that $(\bigoplus^n_{i=1}\pi^*(\mathcal{T}\otimes \mathcal{M}^{\otimes i})\otimes \mathcal{E}_i)$ is a tilting bundle on $X$ that in addition admits an equivariant structure. For this, we take an equivariant ample sheaf $\mathcal{N}$ on $Z$. Such a $\mathcal{N}$ exists according to Lemma 4.2. Now let $\mathcal{M}=\mathcal{N}^{\otimes m}$ for $m\gg 0$. So we have to find a natural number $m\gg0$ such that
\begin{center}
$\mathrm{Ext}^l(\pi^*(\mathcal{T}\otimes\mathcal{M}^{\otimes i})\otimes\mathcal{E}_i,\pi^*(\mathcal{T}\otimes\mathcal{M}^{\otimes j})\otimes \mathcal{E}_j)=0$, for $l>0$.
\end{center}
But this is equivalent to
\begin{center}
$H^l(X,\pi^*(\mathcal{T}\otimes\mathcal{T}^{\vee}\otimes\mathcal{M}^{\otimes (j-i)})\otimes\mathcal{E}_j\otimes\mathcal{E}^{\vee}_i)=0$, for $l>0$.
\end{center}
Applying the Leray spectral sequence (1) to the morphism $\pi\colon X\rightarrow Z$, we obtain
\begin{eqnarray*}
H^r(Z,\mathbb{R}^s\pi_*(\pi^*(\mathcal{T}\otimes\mathcal{T}^{\vee}\otimes\mathcal{M}^{\otimes (j-i)})\otimes\mathcal{E}_j\otimes\mathcal{E}^{\vee}_i))\Longrightarrow\\
H^{r+s}(X,\pi^*(\mathcal{T}\otimes\mathcal{T}^{\vee}\otimes\mathcal{M}^{\otimes (j-i)})\otimes\mathcal{E}_j\otimes\mathcal{E}^{\vee}_i).
\end{eqnarray*}
With the projection formula we find
\begin{center}
$\mathbb{R}^s\pi_*(\pi^*(\mathcal{T}\otimes\mathcal{T}^{\vee}\otimes\mathcal{M}^{\otimes (j-i)})\otimes\mathcal{E}_j\otimes\mathcal{E}^{\vee}_i)\simeq \mathcal{T}\otimes\mathcal{T}^{\vee}\otimes\mathcal{M}^{\otimes (j-i)}\otimes\mathbb{R}^s\pi_*(\mathcal{E}_j\otimes\mathcal{E}^{\vee}_i)$.
\end{center}
Now from Lemma 4.1 we know that $\mathbb{R}^s\pi_*(\mathcal{E}_j\otimes\mathcal{E}^{\vee}_i)$ is non-vanishing only for $s=0$ and $j\geq i$ and that in this case we have $\mathbb{R}^s\pi_*(\mathcal{E}_j\otimes\mathcal{E}^{\vee}_i)\simeq \pi_*(\mathcal{E}_j\otimes\mathcal{E}^{\vee}_i)$. Thus for $j<i$ we get $\mathbb{R}^s\pi_*(\mathcal{E}_j\otimes\mathcal{E}^{\vee}_i)=0$ and therefore 
\begin{center}
$H^r(Z,\mathcal{T}\otimes\mathcal{T}^{\vee}\otimes\mathcal{M}^{\otimes (j-i)}\otimes\mathbb{R}^s\pi_*(\mathcal{E}_j\otimes\mathcal{E}^{\vee}_i))=0$. 
\end{center}
Therefore we find 
\begin{center}
$H^{l}(X,\pi^*(\mathcal{T}\otimes\mathcal{T}^{\vee}\otimes\mathcal{M}^{\otimes (j-i)})\otimes\mathcal{E}_j\otimes\mathcal{E}^{\vee}_i)=0$, 
\end{center}
for $l>0$ by above spectral sequence. It remains the case $j\geq i$. For $j=i$ we have $\mathbb{R}^s\pi_*(\mathcal{E}_i\otimes\mathcal{E}^{\vee}_i)\simeq \pi_*(\mathcal{E}_i\otimes\mathcal{E}^{\vee}_i)\simeq \mathcal{O}_Z$ (see \cite{SAM}, p.5 right after (3.10)). Hence
\begin{eqnarray*}
H^r(Z,\mathcal{T}\otimes\mathcal{T}^{\vee}\otimes\mathcal{M}^{\otimes (i-i)}\otimes\mathbb{R}^s\pi_*(\mathcal{E}_i\otimes\mathcal{E}^{\vee}_i))&\simeq& H^r(Z,\mathcal{T}\otimes \mathcal{T}^{\vee}\otimes \mathcal{O}_Z)\\
&\simeq&\mathrm{Ext}^r(\mathcal{T},\mathcal{T})=0
\end{eqnarray*}
for $r>0$, as $\mathcal{T}$ is a tilting bundle on $Z$. From the above spectral sequence we conclude 
\begin{center}
$H^{l}(X,\pi^*(\mathcal{T}\otimes\mathcal{T}^{\vee}\otimes\mathcal{M}^{\otimes (i-i)})\otimes\mathcal{E}_i\otimes\mathcal{E}^{\vee}_i)=0$,
\end{center} for $l>0$. Finally, it remains the case $j>i$. Again we consider the above spectral sequence and notice that it becomes
\begin{eqnarray*}
H^r(Z, \mathcal{T}\otimes \mathcal{T}^{\vee}\otimes \mathcal{M}^{\otimes (j-i)}\otimes \pi_*(\mathcal{E}_j\otimes\mathcal{E}^{\vee}_i))\Longrightarrow\\
H^{r}(X,\pi^*(\mathcal{T}\otimes\mathcal{T}^{\vee}\otimes\mathcal{M}^{\otimes (j-i)})\otimes\mathcal{E}_j\otimes\mathcal{E}^{\vee}_i). 
\end{eqnarray*}
Since there are only finitely many $\mathcal{E}_i$ and $Z$ is projective, we can choose the natural number $m\gg0$ such that for $\mathcal{M}=\mathcal{N}^{\otimes m}$ we get from the ampleness 
\begin{center}
$H^r(Z,\mathcal{T}\otimes\mathcal{T}^{\vee}\otimes\mathcal{M}^{\otimes (j-i)}\otimes\pi_*(\mathcal{E}_j\otimes\mathcal{E}^{\vee}_i))=0$ for $r>0$.
\end{center}
This finally yields 
\begin{center}
$H^l(X,\pi^*(\mathcal{T}\otimes\mathcal{T}^{\vee}\otimes\mathcal{M}^{\otimes (j-i)})\otimes\mathcal{E}_j\otimes\mathcal{E}^{\vee}_i)=0$ for $l>0$
\end{center} and therefore
\begin{center}
$\mathrm{Ext}^l(\pi^*(\mathcal{T}\otimes\mathcal{M}^{\otimes i})\otimes\mathcal{E}_i,\pi^*(\mathcal{T}\otimes\mathcal{M}^{\otimes j})\otimes \mathcal{E}_j)=0$ for $l>0$.
\end{center} Thus the Ext vanishing holds. The generating property of $\mathcal{R}:=\bigoplus^n_{i=1}\pi^*(\mathcal{T}\otimes \mathcal{M}^{\otimes i})\otimes \mathcal{E}_i$ can be seen as follows:

By assumption, $\mathcal{T}$ is a tilting bundle on $Z$. Now Proposition 3.6 shows that $\mathcal{T}\otimes \mathcal{M}^{\otimes i}$ is a tilting bundle on $Z$ for all $i\in \mathbb{Z}$. In particular, $\mathcal{T}\otimes \mathcal{M}^{\otimes i}$ generates $D(\mathrm{Qcoh}(Z))$ and thus $\langle\mathcal{T}\otimes \mathcal{M}^{\otimes i}\rangle=D^b(Z)$ by Theorem 3.1. Now Lemma 4.3. implies that $\langle\mathcal{R}\rangle=D^b(X)$ and hence $\mathcal{R}$ generates $D(\mathrm{Qcoh}(X))$, again by Theorem 3.1. Together with the Ext vanishing this yields that $\mathcal{R}$ is a tilting bundle on $X$. Now notice that $\mathcal{T}, \mathcal{M}\in \mathrm{Coh}_G(Z)$ and $\mathcal{E}_i\in \mathrm{Coh}_G(X)$. So $\pi^*(\mathcal{T}\otimes \mathcal{M}^{\otimes i})\in \mathrm{Coh}_G(X)$ and thus $\mathcal{R}\in D^b_G(X)$. Theorem 4.4 implies that $\mathcal{R}\otimes k[G]$ is a tilting bundle on $[X/G]$ and completes the proof. 
\end{proof}

\begin{exam}
\textnormal{Let $Z$ be a smooth projective $k$-scheme on which a finite group $G$ acts. Suppose $\mathcal{T}\in \mathrm{Coh}_G(Z)$, considered as an object in $\mathrm{Coh}(Z)$, is a tilting bundle on $Z$. Now let $\mathcal{E}$ be an equivariant locally free sheaf of rank $r$ on $Z$. Then $\pi\colon \mathbb{P}(\mathcal{E})\rightarrow Z$ is a flat $G$-map. As $G$ acts naturally on $\mathbb{P}(\mathcal{E})$, the sheaves $\mathcal{O}_{\mathcal{E}}, \mathcal{O}_{\mathcal{E}}(1),..., \mathcal{O}_{\mathcal{E}}(r-1)$ are objects of $\mathrm{Coh}_G(\mathbb{P}(\mathcal{E}))$. It can be shown that their restriction to any fiber $\pi^{-1}(z)$, $z\in Z$, is a full strongly exceptional collection. Essentially, this follows from \cite{HA}, Exercise 8.4 and \cite{BE}. Theorem 4.5 now states that there exists an equivariant ample sheaf $\mathcal{M}$ on $Z$ such that $(\bigoplus^{r-1}_{i=0}\pi^*(\mathcal{T}\otimes \mathcal{M}^{\otimes (i+1)})\otimes \mathcal{O}_{\mathcal{E}}(i))\otimes k[G]$ is a tilting bundle on $[ \mathbb{P}(\mathcal{E})/G]$. Note that this is proved in \cite{NO} in a direct way.}
\end{exam}
The following fact is well known. For convenience of the reader, we give a proof.
\begin{prop}
Let $X$ be a smooth projective $k$-scheme and suppose $\mathcal{T}$ is a tilting bundle on $[X/G]$ whose indecomposable pairwise non-isomorphic direct summands $\mathcal{L}_i$ are invertible sheaves. Then the set of the $\mathcal{L}_i$ can be ordered in such a way that the reordered set forms a full strongly exceptional collection in $D^b_G(X)$.
\end{prop}
\begin{proof}
According to the Krull--Schmidt Theorem for coherent sheaves \cite{AT} we write $\mathcal{T}=\bigoplus^n_{i=1}\mathcal{L}^{\oplus r_i}_i$, where $\mathcal{L}_i$ are the indecpomposable pairwise non-isomorphic direct summands. As $\langle \mathcal{T}\rangle =D^b_G(X)$, we clearly have $\langle \mathcal{E}_i\rangle_{1\leq i\leq n}=D^b_G(X)$. So the summands form a full collection. From the fact that $\mathrm{Ext}^l_G(\mathcal{T},\mathcal{T})=0$ for $l>0$, we deduce that $\mathrm{Ext}^l_G(\mathcal{L}_i,\mathcal{L}_j)=0$ for any $1\leq i,j\leq n$ and $l>0$. Note that for each $\mathcal{L}_i$ we have $\mathrm{Hom}_G(\mathcal{L}_i,\mathcal{L}_i)=k$. So, according to Definition 3.7, we only need to see that we can order the summands $\mathcal{L}_i$ in such a way that $\mathrm{Hom}_G(\mathcal{L}_r,\mathcal{L}_s)=0$ for any $1\leq s<r\leq n$. Since $\mathrm{Hom}_G(\mathcal{L}_i,\mathcal{L}_i)=k$, then for any $1\leq i,j\leq n$ either $\mathrm{Hom}_G(\mathcal{L}_j,\mathcal{L}_i)=0$ or $\mathrm{Hom}_G(\mathcal{L}_i,\mathcal{L}_j)=0$. Thus, the set $\mathcal{L}_1,...,\mathcal{L}_n$ can be ordered in such a way that the reordered set becomes a full strongly exceptional collection. 
\end{proof}
\begin{prop}
Let $X$ be a smooth projective $k$-scheme on which a finite group $G$ acts. Suppose that $\mathcal{E}_1,...,\mathcal{E}_n\in \mathrm{Coh}_G(X)$, considered as a set of objects in $D^b(X)$, is a full strongly exceptional collection for $D^b(X)$. If we denote by $W_1,...,W_m$ the irreducible representation of $G$, then the collection
\begin{eqnarray}
\{\{\mathcal{E}_i\otimes W_1\}_{1\leq i\leq n}, \{\mathcal{E}_i\otimes W_2\}_{1\leq i\leq n},..., \{\mathcal{E}_i\otimes W_m\}_{1\leq i\leq n}\}
\end{eqnarray} is a full strongly exceptional collection for $D^b_G(X)$.
\end{prop}
\begin{proof}
We have canonical isomorphisms
\begin{center}
$\mathrm{Ext}^l(\mathcal{E}_i\otimes W_p,\mathcal{E}_j\otimes W_q)\simeq \mathrm{Ext}^l(\mathcal{E}_i,\mathcal{E}_j)\otimes \mathrm{Hom}(W_p,W_q)$
\end{center} on $X$. Lemma 2.2 gives 
\begin{center}
$\mathrm{Ext}^l_G(\mathcal{E}_i\otimes W_p,\mathcal{E}_j\otimes W_q)\simeq (\mathrm{Ext}^l(\mathcal{E}_i,\mathcal{E}_j)\otimes \mathrm{Hom}(W_p,W_q))^G$.
\end{center} From Schur's Lemma and the fact that $\mathcal{E}_1,...,\mathcal{E}_n$, considered as a set of objects in $D^b(X)$, is a strongly exceptional collection for $D^b(X)$, it easily follows that (2) is a strongly exceptional collection in $D^b_G(X)$. The generating property of this collection can be seen as follows:
As $\mathcal{T}=\bigoplus^n_{i=1}\mathcal{E}_i\in D^b_G(X)$, considered as an object in $D^b(X)$, is a tilting bundle on $X$, we conclude from Theorem 4.4 that $\bigoplus^m_{i=1}\mathcal{T}\otimes W_i$ is a tilting bundle on $[X/G]$. In particular, $\bigoplus^m_{i=1}\mathcal{T}\otimes W_i$ generates $D_G(\mathrm{Qcoh}(X)$. Note that the compact objects of $D_G(\mathrm{Qcoh}(X)$ are all of $D^b_G(X)$ (see \cite{BR}, p.39). Since $\bigoplus^m_{i=1}\mathcal{T}\otimes W_i$ is compact, Theorem 3.1 yields $\langle \bigoplus^m_{i=1}\mathcal{T}\otimes W_i\rangle=D^b_G(X)$. Thus our collection (2) is full. 
\end{proof}
We are now able to prove the following:
\begin{thm}
Let $\pi\colon X\rightarrow Z$ be a flat $G$-map and $\mathcal{E}_1,...,\mathcal{E}_n$ a set of invertible sheaves in $D^b_G(X)$ such that, considered as a set of objects in $D^b(X)$, for any point $z\in Z$ the collection $\mathcal{E}^z_1=\mathcal{E}_1\otimes \mathcal{O}_{X_z},...,\mathcal{E}^z_n=\mathcal{E}_n\otimes \mathcal{O}_{X_z}$ of the restrictions to the fiber $\pi^{-1}(z)=X_z$ is a full strongly exceptional collection for $D^b(X_z)$. Suppose $\mathcal{T}\in \mathrm{Coh}_G(Z)$ is a tilting bundle on $Z$ whose indecomposable pairwise non-isomorphic direct summands are invertible sheaves. Then there is a full strongly exceptional collection for $D^b_G(X)$.
\end{thm}
\begin{proof}
By assumption $\mathcal{T}$ is a tilting bundle on $Z$ whose indecomposable pairwise non-isomorphic direct summands $\mathcal{L}_1,...,\mathcal{L}_m$ are invertible sheaves. Proposition 4.7 ensures that the set $\mathcal{L}_1,...,\mathcal{L}_m$ can be reordered in such a way that the reordered set forms a full strongly exceptional collection. Let us denote this full strongly exceptional collection by $\mathcal{L}'_1,...,\mathcal{L}'_m$. Now \cite{CDR}, Theorem 2.8 states that there is an ample sheaf $\mathcal{M}$ on $Z$ such that 
\begin{eqnarray}
\{\{\pi^*(\mathcal{L}'_i\otimes \mathcal{M})\otimes \mathcal{E}_1  \}_{1\leq i\leq m},...,\{\pi^*(\mathcal{L}'_i\otimes \mathcal{M}^{\otimes n})\otimes \mathcal{E}_n  \}_{1\leq i\leq m}\} 
\end{eqnarray}
is a full strongly exceptional collection in $D^b(X)$. As the proof of this fact follows exactly the lines of the proof of Theorem 4.5, we see that the ample sheaf $\mathcal{M}$ on $Z$ can be chosen to be equivariant. In this way we get that all members of the collection (3) are objects in $\mathrm{Coh}_G(X)$ and Proposition 4.8 provides us with a full strongly exceptional collection for $D^b_G(X)$.
\end{proof}
\begin{rema}
Theorem 4.9 in connection with Example 3.9 also gives a semiorthogonal decomposition for $D^b_G(X)$.
\end{rema}

{\small MATHEMATISCHES INSTITUT, HEINRICH--HEINE--UNIVERSIT\"AT 40225 D\"USSELDORF, GERMANY}\\
E-mail adress: novakovic@math.uni-duesseldorf.de

\end{document}